\documentclass{article}

\usepackage{amsmath}
\usepackage{amsfonts}
\usepackage{amssymb}
\usepackage{amsthm}
\usepackage{array}
\usepackage{chngpage}
\usepackage{comment}
\usepackage{float}
\usepackage[OT2,T1]{fontenc}
\usepackage{graphicx,caption}
\usepackage[export]{adjustbox}
\usepackage{longtable}
\usepackage{mathrsfs}
\usepackage{wrapfig}
\usepackage{cutwin}
\usepackage{shapepar}
\usepackage{tikz}
\usepackage[lowtilde]{url}
\usetikzlibrary{matrix,arrows}

\newcommand{\F}{\mathbb{F}}

\newcommand{\Q}{\mathbb{Q}}

\newcommand{\Z}{\mathbb{Z}}

\newcommand{\cE}{\mathcal{E}}

\newcommand{\cO}{\mathcal{O}}
\newcommand{\cP}{\mathcal{P}}

\newcommand{\cS}{\mathcal{S}}

\newcommand{\sS}{\mathscr{S}}

\DeclareSymbolFont{cyrletters}{OT2}{wncyr}{m}{n}
\DeclareMathSymbol{\sha}{\mathalpha}{cyrletters}{"58}

\newcommand{\eps}{\varepsilon}

\newcommand{\ds}{\displaystyle}
\newlength{\strutheight}
\newtheorem{theorem}{Theorem}[section]

\newtheorem{conjecture}[theorem]{Conjecture}
\newtheorem{proposition}[theorem]{Proposition}

\author{Alex Cowan\\cowan@math.harvard.edu}
\title{Conjecture: $100\%$ of elliptic surfaces over $\Q$ have rank zero}
\date{}

\begin{document}
\maketitle
\begin{abstract}
\noindent
Based on an equation for the rank of an elliptic surface over $\Q$ which appears in the work of Nagao, Rosen, and Silverman, we conjecture that $100\%$ of elliptic surfaces have rank $0$ when ordered by the size of the coefficients of their Weierstrass equations, and present a probabilistic heuristic to justify this conjecture. We then discuss how it would follow from either understanding of certain $L$-functions, or from understanding of the local behaviour of the surfaces. Finally, we make a conjecture about ranks of elliptic surfaces over finite fields, and highlight some experimental evidence supporting it.
\end{abstract}

\section{Introduction}\label{introduction}
Let $\cE$ be an elliptic surface over $\Q$. Fix a Weierstrass equation
$$\cE: y^2 = x^3 + A(T)x + B(T)$$
with $A(T), B(T) \in \Z[T]$ and $4A(T)^3 + 27B(T)^2 \not\equiv 0$. Given an integer $t$ such that $4A(t)^3 + 27B(t)^2\neq 0$, let $\cE_t$ denote the elliptic curve over $\Q$ with Weierstrass equation
$$\cE_t: y^2 = x^3 + A(t)x + B(t).$$
Given an elliptic curve $E/\Q$ and a prime $p$, define $a_p(E)$ to be the $p^{\text{th}}$ coefficient of the $L$-function attached to $E$.\\
\\
It was conjectured by Nagao \cite{nagao} that
\begin{equation}\label{nagaoeqn}
  \text{rank }\cE(\Q(T)) = -\lim_{X\to\infty}\frac{1}{X}\sum_{p<X}\sum_{t=1}^pa_p(\cE_t)\frac{\log p}{p}.
\end{equation}
Rosen and Silverman \cite{rosen_silverman} proved that (\ref{nagaoeqn}) held if one assumed Tate's conjecture and that a certain $L$-function didn't vanish on the right edge of its critical strip. In particular, (\ref{nagaoeqn}) holds unconditionally for rational elliptic surfaces.\\
\\
In light of (\ref{nagaoeqn}), we think it's likely that the average rank of elliptic surfaces over $\Q$ is $0$, essentially because the average value of $a_p(\cE_t)$ should be $0$. We give a more thorough justification for this belief in section \ref{justification}.\\
\\
There are several ways the belief that the average rank of elliptic surfaces is $0$ can be formulated as a precise conjecture. For convenience, we introduce the following framework for discussing possible formulations of such a conjecture.\\
\\
Let $\cS(M)$ with $m \in \Z_{>0}$ be a sequence of subsets of the set of all elliptic surfaces over $\Q$, with the properties that $\cS(M)$ is finite for every $M$, and $\cS(M) \subseteq \cS(M')$ whenever $M \leq M'$. We'll say that $\cS = \cup_{M=1}^\infty \cS(M)$ is a \textit{family} of elliptic surfaces, and the filtration given by $M$ is an \textit{ordering} of that family. This parallels the language used when discussing statistics of ranks of elliptic curves, where, for example, people might discuss ``the family of quadratic twists of a fixed elliptic curve, ordered by conductor''.\\
\\
Define the \textit{average rank of the family $\cS$} to be the quantity
  $$\lim_{M\to\infty}\frac{1}{\#\cS(M)}\sum_{\cE \in \cS(M)} \text{rank }\cE(\Q(T)).$$
We believe that the average rank of many families $\cS$ will be $0$. In formulating conjecture \ref{avgrank0} below, we pick a specific family for the sake of concreteness, but there are many other reasonable choices.\\% For example, below we'll propose ordering the polynomials $A(T)$ and $B(T)$ by their Mahler measure, but ordering by them by the maximum of the absolute values of their coefficients would also be reasonable, and we'll take $A(T)$ and $B(T)$ to be roughly the same size, but it would also be reasonable to pick a formulation that causes $A(T)^3$ and $B(T)^2$ to be of similar sizes instead.\\
\\
Conversely, there are several examples in the literature of families of elliptic surfaces over $\Q$ which were constructed to have strictly positive rank \cite{elkies} \cite{fermigier}. Average ranks of the elliptic curves coming from specializations of elliptic surfaces have been studied as well \cite{bdd} \cite{fermigier} \cite{silverman}, and in some situations the average ranks of the specializations can exhibit rich and possibly unexpected behaviour.\\
%In particular, in \cite{silverman}, Silverman writes that the computations in \cite{fermigier} might suggest that, for a wide variety of elliptic surfaces $\cE$,
%$$\frac{1}{2T}\sum_{|t|<T}\text{rank }\cE_t(\Q) \sim \frac{1}{2} + \text{rank }\cE(\Q) + \delta_\cE$$
%for some $\delta_\cE > 0$. This is incompatible with the belief that the average rank of the family $\{\cE_t\}$ of ellitpic curves having average rank $\frac{1}{2}$, as well as certain ``rules of thumb'' of the Katz-Sarnak philosophy \cite{katz_sarnak}, but it is known that certain elliptic surfaces exhibit biases in the root numbers of specializations \cite{bdd}. It might be worthwhile to revisit these computations, now that much more powerful hardware is available.\\
\\
To formulate conjecture \ref{avgrank0} below, we first define, for positive integers $d$ and $M$, the set
$$\cP_d(M) = \left\{p \in \Z[T]\,:\, \text{deg}(p) = d,\, \mu(p) < M\right\},$$
where $\mu(p)$ is the Mahler measure of $p$ (using the height of $p$ would also be reasonable here). Set
$$\cS_{m,n}(M) := \{\cE: y^2 = x^3 + A(T)x + B(T)\,:\, A\in \cP_m(M^2),\, B\in\cP_n(M^3),\, 4A(T)^3 + 27B(T)^2 \not\equiv 0\}.$$
Then we believe the average rank of the family $\cS = \cS_{m,n}$ will be $0$ whenever $m$ and $n$ are both positive, i.e.
\begin{conjecture}\label{avgrank0} For any fixed positive integers $m$ and $n$, we have
  $$\lim_{M\to\infty}\frac{1}{\#\cS_{m,n}(M)}\sum_{\cE \in \cS_{m,n}(M)}\text{\emph{rank }}\cE(\Q(T)) = 0.$$
\end{conjecture}
\noindent
In section \ref{analyticproofsection} we discuss the main obstacles for proving conjecture \ref{avgrank0} using the analytic framework in \cite{rosen_silverman}. In section \ref{finitefieldsection} we outline an approach to conjecture \ref{avgrank0} based on investigating statistics of ranks of elliptic surfaces over finite fields.

\section{Acknowledgments}
We thank Noam Elkies, Bjorn Poonen, and Michael Snarski for helpful discussions. This work was supported by grant 550031 from the Simons Foundation.

\section{Probabilistic heuristics}\label{justification}
For a fixed prime $p$, Birch \cite{birch} gave all moments of the distribution of $a_p(E)$ when $E$ is chosen by selecting a (possibly singular) Weierstrass equation with coefficients in $\F_p$ uniformly at random. Let $(A_{p,t})$ be a sequence of independent random variables indexed by a prime $p$ and a positive integer $t$, with the property that, for every $t$, the random variable $A_{p,t}$ has the same distribution as the values $a_p(E)$ for $E$ a Weierstrass equation in $\F_p$ chosen uniformly at random, as in Birch's work. We highlight the following property of the sequence $(A_{p,t})$:
\begin{proposition}\label{threeseries} For any $\eps > 0$, the series
  $$\frac{1}{X^{\frac{1}{2}+\eps}}\sum_{p < X}\sum_{t = 1}^p A_{p,t}\frac{\log p}{p}$$
  converges to $0$ with probability $1$ as $X$ goes to infinity.
\end{proposition}
\begin{proof} By construction we have
  $$\text{Var}\!\left(\frac{1}{X^{\frac{1}{2}+\frac{\eps}{2}}}A_{p,t}\frac{\log p}{p}\right) \ll \frac{(\log p)^2}{X^{1+\eps}p}.$$
  Thus
  $$\text{Var}\!\left(\frac{1}{X^{\frac{1}{2}+\frac{\eps}{2}}}\sum_{p < X}\sum_{t = 1}^p A_{p,t}\frac{\log p}{p}\right) \ll \frac{1}{X^{1+\eps}}\sum_{p<X}(\log p)^2 \ll 1.$$
  By Kolmogorov's three-series theorem, as stated in \cite[theorem 2.5.8]{durrett} for example, it follows that the series
  $$\frac{1}{X^{\frac{1}{2}+\frac{\eps}{2}}}\sum_{p < X}\sum_{t = 1}^p A_{p,t}\frac{\log p}{p}$$
  converges almost surely. Hence the series
  $$\frac{1}{X^{\frac{1}{2}+\eps}}\sum_{p < X}\sum_{t = 1}^p A_{p,t}\frac{\log p}{p} = \frac{1}{X^{\frac{\eps}{2}}} \left( \frac{1}{X^{\frac{1}{2}+\frac{\eps}{2}}}\sum_{p < X}\sum_{t = 1}^p A_{p,t}\frac{\log p}{p}\right)$$
  converges to $0$ almost surely.
\end{proof}
\noindent
Proposition \ref{threeseries} is relevant because it suggests that, if we believe that the values $a_p(\cE_t)$ are distributed in the same way $a_p(E)$ is for elliptic curves $E/\F_p$ chosen uniformly at random, then we should expect that $100\%$ of the time the quantity in (\ref{nagaoeqn}) is $0$.\\
\\
There is evidence to support the belief that the values $a_p(\cE_t)$ will be distributed in this way. If, instead of fixing $p$ and letting $t$ vary initially, we fix $t$ and let $p$ vary, then, for all $\eps > 0$ the series
\begin{equation}\label{asconvofap}
  \frac{1}{X^\eps}\sum_{p < X}A_{p,t}\frac{\log p}{p}
\end{equation}
converges to $0$ as $X$ goes to infinity with probability $1$. This suggests that, for a fixed elliptic curve $E_t/\Q$, we should expect that
\begin{equation}\label{heathbrowneqn}
  \sum_{p < X} a_p(E_t)\frac{\log p}{p} \ll X^\eps
\end{equation}
for all $\eps > 0$, because we should expect that the reductions $E_t/\F_p$ will be distributed uniformly at random. The bound (\ref{heathbrowneqn}) was proven by Heath-Brown \cite{heath-brown}, so the reductions $E_t/\F_p$ do indeed behave uniformly randomly in this case.\\
\\
The popular belief that the average rank of elliptic curves over $\Q$ is $\frac{1}{2}$ can also support the idea that these sums of $a_p$ values behave randomly in the way described. The Birch and Swinnerton-Dyer conjecture (BSD) connects the rank of an elliptic curve $E/\Q$ to the coefficients $a_p(E)$ of its $L$-function. The original formulation of BSD was
$$\prod_{p < X}\frac{p+1 - a_p(E)}{p} \sim C_E(\log X)^{\text{rank } E(\Q)}$$
for some constant $C_E$ which depends on $E$. See work of Goldfeld \cite{goldfeld}, K. Conrad \cite{conrad}, and Kuo-R. Murty \cite{kuo_murty} for some treatment of this specific formulation of BSD. One consequence of this form of BSD is \cite{rubinstein}
\begin{equation}\label{rubeqn}
  \text{rank }E(\Q) = \frac{1}{2} - \frac{1}{\log X}\int_1^X\frac{1}{x}\sum_{p<x}a_p(E)\log p\,\frac{dx}{x} + \cO\Big(\frac{1}{\log X}\Big).
\end{equation}
It is widely believed that many families of elliptic curves over $\Q$ have average rank $\frac{1}{2}$. This belief suggests that the quantity
\begin{equation}\label{rubquantity}
  \frac{1}{x}\sum_{p<x}a_p(E)\log p
\end{equation}
appearing in equation (\ref{rubeqn}) should average to $0$ in those families. However, because the error term in (\ref{rubeqn}) depends on $E$, even knowing that the average rank of elliptic curves was $\frac{1}{2}$ wouldn't be enough to conclude that the quantity (\ref{rubquantity}) averages to $0$. While it would be surprising if the error terms in (\ref{rubeqn}) did not average to $0$ as well, controlling these error terms is difficult, and is the main obstacle in proving anything about average ranks from this perspective.

\section{Obstacles for analytic proofs}\label{analyticproofsection}
In this section we discuss what obstacles exist which make proving conjecture \ref{avgrank0} difficult within the framework established in \cite{rosen_silverman}. In proving Nagao's conjecture, Rosen and Silverman first prove an analytic version of his conjecture:
\begin{equation}\label{annagao}
  \underset{s = 2}{\text{Res}}\,\sum_p\sum_{t=1}^p a_p(\cE_t)\frac{\log p}{p^s} = -\text{rank }E(\Q(T)).
\end{equation}
To do this, they introduce the following notation:
\begin{itemize}
\item $\ds{L_2(\cE/\Q,s)}$ is the Hasse-Weil $L$-function of $\cE/\Q$ attached to $\ds{H^2_{\acute{e}t}(\cE/\bar{\Q},\Q_\ell)}$.
\item $\ds{\text{NS}(\cE/\bar{\Q})}$ is the N\'{e}ron-Severi group of $\cE/\bar{\Q}$.
\item $\ds{\sS}$ is the trivial part of $\text{NS}(\cE/\bar{\Q}) \otimes \Q$, generated by the image of the zero section and by all components of all fibers.
\item $\ds{\sS_\ell(1)}$ is the Tate twist $\sS\otimes T_\ell(\bar{\Q}^*)$ of the $\text{Gal}(\bar{\Q}/\Q)$-module $\sS$.
\item $\ds{L(\sS_\ell(1),s)}$ is the Artin $L$-function attached to the representation $\sS_\ell(1)$.
\item $\ds{N(\cE/\Q,s) = \frac{L_2(\cE/\Q,s)}{L(\sS_\ell(1),s)}}.$
\end{itemize}
Then, assuming the Tate conjecture, Silverman and Rosen prove that both sides of (\ref{annagao}) are equal to
$$\underset{s = 2}{\text{Res}}\,\frac{d}{ds}\log N(\cE/\Q,s).$$
The conjecture (\ref{nagaoeqn}) then follows from standard analytic techniques applied to the series $\frac{d}{ds}\log N(\cE/\Q,s)$.\\
\\
We now discuss what might be involved in a proof of conjecture \ref{avgrank0} using these analytic techniques. Suppose that there exist constants $\delta_\cE > 0$ and $C_\cE$ such that
$$\left|\frac{1}{X}\sum_{p < X}\sum_{t = 1}^p a_p(\cE_t)\frac{\log p}{p} + \text{rank }\cE(\Q(T))\right| < C_\cE X^{-\delta_\cE},$$
as is typical in this kind of setting. To prove conjecture \ref{avgrank0}, it would be sufficient to find, for every elliptic surface $\cE \in \cS$, a real number $X_\cE > 0$ depending only on $\cE$, such that both
\begin{equation}\label{apcond}
  \lim_{M\to\infty}\frac{1}{\#\cS(M)}\sum_{\cE \in \cS(M)}\frac{1}{X_{\cE}}\sum_{p < X_{\cE}}\sum_{t = 1}^p a_p(\cE_t)\frac{\log p}{p} = 0
\end{equation}
and
\begin{equation}\label{errorcond}
  \lim_{M\to\infty}\frac{1}{\#\cS(M)}\sum_{\cE \in \cS(M)}C_{\cE}X_{\cE}^{-\delta_{\cE}} = 0
\end{equation}
simultaneously. Condition (\ref{apcond}) requires that $M$ be ``large'' relative to the values $X_{\cE}$ so that the averages
$$\frac{1}{\#\cS(M)}\sum_{\cE \in \cS(M)}a_p(\cE_t)$$
for fixed $p$ and $t$ are small. Condition (\ref{errorcond}), on the other hand, would like for $M$ to be ``small'' relative to the values $X_{\cE}$. The difficulty, then, is to find some choice of $X_{\cE}$'s such that both of these conditions hold at once. We expect that showing that there is such a choice will be difficult. Condition (\ref{apcond}) might require a P\'{o}lya-Vinogradov style result, where one shows that there are no ``conspiracies'' among the values $a_p(\cE_t)$ that might cause them to behave differently than random variables. Condition (\ref{errorcond}) might require subconvexity bounds, and the number $\delta_{\cE}$ will depend on the locations of the zeroes of $N(\cE/\Q,s)$.

\section{An approach via finite fields}\label{finitefieldsection}
Let $\cS_{\ell;m,n}$ denote the set of elliptic surfaces $\cE/\F_\ell: y^2 = x^3 + A(T)x + B(T)$ with $\deg(A) = m$ and $\deg(B) = n$. This set is finite. Let $\rho_\ell(m,n)$ denote the proportion of elliptic surfaces in $\cS_{\ell;m,n}$ which have positive rank.
\begin{conjecture}\label{finitefieldconjecture}
  For every prime $\ell_0$, and every pair of integers $m_0, n_0 > 0$,
  $$\lim_{\ell\to\infty}\rho_{\ell}(m_0,n_0) = \lim_{m\to\infty}\rho_{\ell_0}(m,n_0) = \lim_{n\to\infty}\rho_{\ell_0}(m_0,n) = \frac{1}{2}.$$
\end{conjecture}
\noindent This conjecture, beyond being interesting in its own right, provides an approach for proving conjecture \ref{avgrank0}. See \cite{lauder} for experimental evidence towards conjecture \ref{finitefieldconjecture}, where $\rho_\ell(m,n)$ is estimated computationally for $\ell = 7$, $n = 6, 12, 18, 24, 30$, and $m \leq n/2$.\\
\\
Let $N$ be a squarefree positive integer. Let $\cS^{(N)}(M)$ denote the subset of $\cS(M)$ for which $\text{gcd}(N, 4A(T)^3 + 27B(T)^2) = 1$. Then, for any $m$, $n$, and $M$,
\begin{equation}\label{crt}
  \frac{\#\{\cE \in \cS^{(N)}(M)\,:\,\cE/\F_\ell \text{ has positive rank for all } \ell|N\}}{\#\cS^{(N)}(M)} = \prod_{\ell|N}\rho_\ell(m,n) + \cO(M^{-1})
\end{equation}
by the Chinese remainder theorem, where $\cE/\F_\ell$ denotes the reduction mod $\ell$ of $\cE/\Q$.\\
\\
If $\cE/\Q$ has positive rank, then either the reduction $\cE/\F_\ell$ has positive rank, or the kernel of the reduction $\cE(\Q(T)) \to \cE(\F_\ell(T))$ is of finite index in $\cE(\Q(T))$. If this kernel was never of finite index then conjecture \ref{avgrank0} would follow from conjecture \ref{finitefieldconjecture} (as well as much weaker versions of this conjecture), via observation (\ref{crt}). The kernel of the reduction $\cE(\Q(T)) \to \cE(\F_\ell(T))$ is of finite index in $\cE(\Q(T))$ occasionally, but presumably not nearly enough for this approach to fail. However, proving as much seems difficult. The generators of $\cE(\Q(T))$ will map to the identity of $\cE(\F_\ell(T))$ if their denominators are divisible by $\ell$, so one is naturally lead to investigate the dependence of the height of the generators of $\cE(\Q(T))$ on the size of the coefficients of the Weierstrass model of $\cE/\Q$.
\bibliographystyle{plain}
\bibliography{avgrankbib}{}

\end{document}